\documentclass[12pt,a4paper]{article}
\usepackage{amsmath}
\usepackage{amssymb}
\usepackage{amsthm}
\usepackage{amsfonts}
\usepackage{enumerate}
\usepackage[pdfstartview=FitH]{hyperref}
\makeatletter
\def\tank#1{\protected@xdef\@thanks{\@thanks
 \protect\footnotetext[0]{#1}}}
\def\bigfoot{

 \@footnotetext}
\makeatother

\allowdisplaybreaks

\newtheorem{theorem}{Theorem}[section]
\newtheorem{lem}{Lemma}[section]
\newtheorem{prp}[theorem]{Proposition}
\newtheorem{thm}[theorem]{Theorem}

\newtheorem{dfn}{Definition}[section]

\newtheorem{remark}{Remark}[section]

\title{Anticipating stochastic equation of two-dimensional second grade fluids}

\footnotesize{\author{
 Shijie Shang$^{1,}$\thanks{ssjln@mail.ustc.edu.cn}\\
 {\em $^1$ School of Mathematical Sciences,}\\
 {\em University of Science and Technology of China,}\\
 {\em Hefei, 230026, China}\\
}
\date{}

\begin{document}
\maketitle
\noindent \textbf{Abstract:}
In this paper, we consider a stochastic model of incompressible second grade fluids on a bounded domain of $\mathbb{R}^2$ driven by linear multiplicative Brownian noise
with anticipating initial conditions. The existence and uniqueness of the solutions are established.


\vspace{4mm}


\vspace{3mm}
\noindent \textbf{Key Words:}
Second grade fluids;
Malliavin calculus;
Anticipating Stratonovich integral;
Skorohod integral

\numberwithin{equation}{section}
\section{Introduction}

In this article, we investigate the existence and uniqueness of solutions of the following anticipating stochastic equation of second grade fluids:
\begin{align}\label{1.a}
\left\{
\begin{aligned}
& d(u-\alpha \Delta u)+ \Big(-\nu \Delta u+{\rm curl}(u-\alpha \Delta u)\times u+\nabla\mathfrak{P}\Big)\,dt \\
&\quad = F(u,t)\,dt+(u-\alpha \Delta u)\circ \sigma dW,\quad \rm{ in }\ \mathcal{O}\times(0,T], \\
&\begin{aligned}
& {\rm{div}}\,u=0 \quad &&\rm{in}\ \mathcal{O}\times(0,T]; \\
& u=0  &&\rm{in}\ \partial \mathcal{O}\times[0,T]; \\
& u(0)=\xi  &&\rm{in}\ \mathcal{O},
&\end{aligned}
\end{aligned}
\right.
\end{align}

\noindent where $\mathcal{O}$ is a bounded domain of $\mathbb{R}^2$, simply-connected and open, with boundary $\partial \mathcal{O}$ of class $\mathcal{C}^{3,1}$. $u=(u_1,u_2)$ and $\mathfrak{P}$ represent the random velocity and modified pressure, respectively. $\alpha,\sigma$ are positive constants and $\nu$ is the kinematic viscosity. $W$ is a one-dimensional standard Brownian motion defined on a complete filtered probability space $(\Omega,\mathcal{F},P)$ with the augmented Brownian filtration $\{\mathcal{F}_t\}_{t\geq0}$. $\xi$ is an $\mathcal{F}_T$-measurable random variable. The fluid is driven by external forces $F(u,t)\,dt$ and the noise $(u-\alpha \Delta u)\circ\sigma dW$, where the stochastic integral is understood in the sense of anticipating Stratonovich integrals.

%

\vskip 0.3cm
We refer the reader to \cite{1974-Dunn-p191-252,1995-Dunn-p689-729,1979-Fosdick-p145-152,1997-Cioranescu-p317-335,1984-Cioranescu-p178-197} for a comprehensive theory of the second grade fluids. These fluids are non-Newtonian fluids of differential type,  they are admissible models of slow flow fluids such as industrial fluids, slurries, polymer melts, etc. They also have interesting connections with other fluid models, see \cite{2016-Arada-p2557-2586,2003-Busuioc-p1119-1119,1999-Busuioc-p1241-1246}.
For researchs on stochastic models of 2D second grade fluids, we refer to \cite{2010-Razafimandimby-p1-47,2012-Razafimandimby-p4251-4270,2017-Shang-p-,2016-Zhai-p1-28,2016-Wang-p196-213,2017-Shang-p-a}.

\vskip 0.3cm
The consideration of the anticipating initial value is based on several aspects: random measurement errors, the stationary point of the stochastic dynamical system, substitution formulas of anticipating Stratonovich integrals. For more details, we refer to Mohammed and Zhang \cite{2013-Mohammed-p1380-1408}. The difficulty in directly proving such a substitution theorem is that Kolmogorov continuity theorem fails within our infinite-dimensional setting. 
To solve this anticipating problem (\ref{1.a}), we proceed with the following steps: firstly, we develop a simple chain rule of Malliavin derivative of Hilbert space-valued random variables and establish a product rule for the Skorohod integrals, see Lemma \ref{39.1.I} and Proposition \ref{39.5.I}; secondly, we use Galerkin approximations to show that the solution of (\ref{1.a}) with deterministic initial value is Mallivin differentiable, see Proposition \ref{40.I}; finally, combining the previous two steps, we easily obtain our main results.
We believe that this method can also be used to solve the problem with anticipating initial value and linear multiplicative noise for more general framework of SPDE.

%

\vskip 0.3cm

The organization of this paper is as follows. In Section 2, we introduce some preliminaries and notations.
In Section 3, we formulate the hypotheses and state our main results. Section 4 is devoted to the proof of the main results.

\vskip 0.3cm

Throughout this paper, $C,C(T),C(T,N)...$ are positive constants depending on some parameters $T,N,...$, whose value may be different from line to line.


\section{Preliminaries}

In this section, we will introduce some functional spaces, preliminaries and notations.
\vskip 0.2cm
For $p\geq 1$ and $k\in\mathbb{N}$, we denote by $L^p(\mathcal{O})$
and $W^{k,p}(\mathcal{O})$ the usual $L^p$ and Sobolev spaces over $\mathcal{O}$ respectively, and write $H^k(\mathcal{O}):=W^{k,2}(\mathcal{O})$.
We write $\mathbb{X}=X\times X$ for any vector space $X$.
The set of all divergence free and infinitely differentiable functions in $\mathcal{O}$ is denoted by $\mathcal{C}$.
$\mathbb{V}$ (resp. $\mathbb{H}$) is the completion of $\mathcal{C}$ in $\mathbb{H}^1(\mathcal{O})$ (resp. $\mathbb{L}^2(\mathcal{O})$),
Let $((u,v)):=\int_\mathcal{O}\nabla u\cdot\nabla vdx$, where $\nabla$ is the gradient operator.
%
%
%
%
Denote $\|u\|:=((u,u))^{\frac{1}{2}}$.
We endow the space $\mathbb{V}$ with the norm $|\cdot|_{\mathbb{V}}$ generated by the following inner product
\[
(u,v)_\mathbb{V}:=(u,v)+\alpha ((u,v)),\quad \text{for any } u,v\in\mathbb{V},
\]
where $(\cdot,\cdot)$ is the inner product in $\mathbb{L}^2(\mathcal{O})$(in $\mathbb{H}$).
%
We also introduce the following space
\[
\mathbb{W}:=\big\{u\in\mathbb{V}: {\rm curl}(u-\alpha\Delta u)\in L^2(\mathcal{O})\big\},
\]
and endow it with the semi-norm $|\cdot|_{\mathbb{W}}$ generated by the scalar product
\begin{align*}
(u,v)_\mathbb{W}:=\big({\rm curl}(u-\alpha\Delta u),{\rm curl}(v-\alpha\Delta v)\big).
\end{align*}
In fact, $\mathbb{W}=\mathbb{H}^3(\mathcal{O})\cap\mathbb{V}$, and this semi-norm $|\cdot|_{\mathbb{W}}$ is equivalent to the usual norm in $\mathbb{H}^3(\mathcal{O})$, the proof can be found in \cite{1984-Cioranescu-p178-197,1997-Cioranescu-p317-335}.
%
%
%

\vskip 0.3cm

Identifying the Hilbert space $\mathbb{V}$ with its dual space $\mathbb{V}^*$, via the Riesz representation, we consider the system (\ref{1.a}) in the framework of Gelfand triple:
$\mathbb{W}\subset \mathbb{V}\subset\mathbb{W}^*$.
%
\noindent We also denote by $\langle \cdot,\cdot\rangle$ the dual relation between $\mathbb{W}^*$ and $\mathbb{W}$ from now on.

Because the injection of $\mathbb{W}$ into $\mathbb{V}$ is compact, there exists a sequence $\{e_i\}$ of elements of $\mathbb{W}$ which forms an orthonormal basis in $\mathbb{W}$, and an orthogonal system in $\mathbb{V}$, such that
\begin{align}\label{Basis}
(u,e_i)_{\mathbb{W}}=\lambda_i(u,e_i)_{\mathbb{V}},\quad\text{for any }u\in\mathbb{W},
\end{align}
where $0<\lambda_i\uparrow\infty$. Since $\partial \mathcal{O}$ is of class $\mathcal{C}^{3,1}$, Lemma 4.1 in \cite{1997-Cioranescu-p317-335} implies that
\begin{align}\label{regularity of basis}
e_i\in \mathbb{H}^4(\mathcal{O}),\ \ \forall\,i\in\mathbb{N}.
\end{align}

\vskip 0.3cm
Define the Stokes operator by
\begin{align*}
Au:=-\mathbb{P}\Delta u,\quad\forall\,u\in D(A)=\mathbb{H}^2(\mathcal{O})\cap\mathbb{V},
\end{align*}
where $\mathbb{P}:\mathbb{L}^2(\mathcal{O})\longrightarrow\mathbb{H}$ is the usual Helmholtz-Leray projection.
Set $\widehat{A}:=(I+\alpha A)^{-1}A$, it follows from \cite{2017-Shang-p-a} that $\widehat{A}$ is a
continuous linear operator from $\mathbb{W}$ onto itself, moreover,
\begin{gather}
\label{A transform 01}(\widehat{A}u,v)_\mathbb{V}=(Au,v)=((u,v)), \quad \forall\,u\in\mathbb{W},\ v\in\mathbb{V}.
\end{gather}

Define the bilinear operator $\widehat{B}(\cdot\,,\cdot):\ \mathbb{W}\times\mathbb{V}\longrightarrow\mathbb{W}^*$ by
\begin{align}\label{definition of B hat}
\widehat{B}(u,v):=(I+\alpha A)^{-1}\mathbb{P}\big({\rm curl}(u-\alpha \Delta u)\times v\big).
\end{align}
For simplicity, we write $\widehat{B}(u):=\widehat{B}(u,u)$. We have the following estimates which can be found in \cite{2012-Razafimandimby-p4251-4270}:
\begin{align}\label{B inequalities}
\begin{aligned}
&|\widehat{B}(u,v)|_{\mathbb{W}^*}\leq C|u|_\mathbb{W}|v|_\mathbb{V},\quad\forall\,u\in\mathbb{W},\ v\in\mathbb{V}\\
&|\widehat{B}(u,u)|_{\mathbb{W}^*}\leq C|u|^2_\mathbb{V},\quad\forall\,u\in\mathbb{W},\\
&\langle\widehat{B}(u,v),v\rangle=0, \quad\forall\,u, v\in\mathbb{W},\\
&\langle\widehat{B}(u,v),w\rangle=-\langle\widehat{B}(u,w),v\rangle,\quad\forall\,u, v, w\in\mathbb{W}.
\end{aligned}
\end{align}

\vskip 0.3cm
Finally, we introduce some notations about Malliavin calculus (see e.g. \cite{2006-Nualart-p-}).
Let $V$ be a real separable Hilbert space, $p\geq 1$,
we denote by $\mathcal{D}^{1,p}(V)$ the Malliavin Sobolev space of all $\mathcal{F}_T$-measurable and Malliavin differentiable $V$-valued random variables with Malliavin derivatives having $p$th-order moments.
The Malliavin derivative of $\xi$ will be a stochastic process denoted by $\{\mathcal{D}_r \xi, 0\leq r\leq T\}$.
$\mathcal{L}^{1,2}(V)$ is the class of $V$-valued processes $u\in L^2([0,T]\times\Omega)$ such that $u(t)\in\mathcal{D}^{1,2}(V)$ for almost all $t$, and there exists a measurable version of the two-parameter process $\mathcal{D}_s u(t)$ verifying $E\int_0^T\int_0^T\left\|\mathcal{D}_s u(t)\right\|_{V}^2\,dsdt<\infty$. Note that $\mathcal{L}^{1,2}(V)$ is isomorphic to $L^2([0,T];\mathcal{D}^{1,2}(V))$.
Let $X\in\mathcal{L}^{1,2}(V)$, we denote by $\mathcal{D}^{+}X$ and $\mathcal{D}^{-}X$ the element of $L^1([0,T]\times\Omega;V)$ satisfying
\begin{align}
\label{D+}\lim_{n\rightarrow\infty}\int_0^T\sup_{s<t\leq (s+1/n)\wedge T}E\|\mathcal{D}_s X_t-(\mathcal{D}^{+}X)_s\|_{V}\,ds=0, \\
\label{D-} \lim_{n\rightarrow\infty}\int_0^T\sup_{(s-1/n)\vee 0\leq t<s}E\|\mathcal{D}_s X_t-(\mathcal{D}^{-}X)_s)\|_{V}\,ds=0 ,
\end{align}
respectively. We denote by $\mathcal{L}^{1,2}_{1}(V)$ the class of processes in $\mathcal{L}^{1,2}(V)$ such that both (\ref{D+}) and (\ref{D-}) hold.
From now on, for $X\in\mathcal{L}^{1,2}_{1}(V)$ we write $(\nabla X)_t:=(\mathcal{D}^{+}X)_t + (\mathcal{D}^{-}X)_t$, and the Fr\'{e}chet derivative is denoted by $\mathbb{D}$.
Let $\mathcal{X}$ denote a class of random variables (or processes), we say that $\xi\in\mathcal{X}_{loc}$ if there exists a sequence of $\{(\Omega_n,\xi^n),n\geq 1\}\subset\mathcal{F}\times \mathcal{X}$ such that $\Omega_n \uparrow \Omega$ and $\xi=\xi^n$ a.s. on $\Omega_n$.

\section{Hypotheses and results}


Let $F:\mathbb{V}\times[0,T]\rightarrow\mathbb{V}$
be a given measurable map. We assume that:

\noindent {\bf(F1)} For any $t\in[0,T]$,
\begin{gather*}
F(0,t)=0,\\
|F(u_1,t)-F(u_2,t)|_\mathbb{V}\leq C_F|u_1-u_2|_\mathbb{V},\quad\forall\,u_1,u_2\in\mathbb{V},
\end{gather*}
where $C_F$ is a constant. In particular, we have
$|F(u,t)|_\mathbb{V}\leq
C_F|u|_\mathbb{V}, \ \forall\,u\in\mathbb{V}, t\in[0,T]$.

\noindent {\bf(F2)} $F$ is Fr\'{e}chet differentiable with respect to the first variable, and the Fr\'{e}chet derivative $\mathbb{D}F:\mathbb{V}\times[0,T] \rightarrow L(\mathbb{V})$ is continuous with respect to the first variable.

\vskip 0.3cm

Set
\[
\widehat{F}(u,t):=(I+\alpha A)^{-1} F(u,t).
\]
Applying $(I+\alpha A)^{-1}$ to the equation (\ref{1.a}), we see that (\ref{1.a}) is equivalent to the stochastic evolution equation:
\begin{align}\label{Abstract}
\left\{
 \begin{aligned}
 & du(t)+\nu \widehat{A}u(t)dt+\widehat{B}\big(u(t),u(t)\big)dt=\widehat{F}\big(u(t),t\big)+u(t) \circ\sigma dW(t), \\
 & u(0)=\xi\quad\text{in}\ \mathbb{W}.
 \end{aligned}
\right.
\end{align}
where the stochastic integral is the anticipating Stratonovich integral.

\begin{dfn}\label{Def 01}
A $\mathbb{V}$-valued continuous and $\mathbb{W}$-valued weakly continuous stochastic process $u$ is called a solution of the system (\ref{1.a}), if the following two conditions hold:

\noindent (1) $u\in \mathcal{L}^{1,2}_{1,loc}(\mathbb{V})$;

\noindent (2) for any $t\in[0,T]$, the following equation holds in $\mathbb{W}^*$ $P$-a.s.:
\begin{align*}
& u(t)+\nu\int_0^t\widehat{A}u(s)\,ds+\int_0^t \widehat{B}\big(u(s),u(s)\big)\,ds
=
\xi+\int_0^t\widehat{F}(u(s),s)\,ds+\int_0^t u(s) \circ\sigma dW(s).
\end{align*}

\end{dfn}

\begin{remark}
To describe the class of anticipating Stratonovich integrable processes, the space $\mathcal{L}^{1,2}_{1,loc}$ is often used. If $u=\{u(s),0\leq s\leq T\}\in\mathcal{L}^{1,2}_{1,loc}$, then $u{\bf I}_{[0,t]}$ is also Stratonovich integrable for all $0< t\leq T$. Moreover, this space has nice relationship between the Stratonovich and the Skorohod integrals(see Theorem 3.1.1 in \cite{2006-Nualart-p-}),
in particular, we have
\begin{align}\label{S-integral and I-integral}
\int_0^t u(s) \circ\sigma dW(s)=\int_0^t \sigma u(s)\,dW(s)+\frac{\sigma}{2}\int_0^t(\nabla u)_s\,ds, \quad \forall\,t\in[0,T].
\end{align}

\end{remark}

Now we can state the main result of this paper.

\begin{thm}\label{56.I}
Assume that {\bf(F1)} and {\bf(F2)} hold, $\xi$ is a $\mathbb{W}\cap \mathbb{H}^4(\mathcal{O})$-valued $\mathcal{F}_T$-measurable random variable, and $\xi\in\mathcal{D}_{loc}^{1,2}(\mathbb{W})$, then there exists a unique solution to the equation (\ref{Abstract}).

\end{thm}

\section{Proof of Theorem \ref{56.I}}
We start with a lemma on a simple chain rule of Malliavin derivative of Hilbert-space valued random variables; next, we establish a product rule for the Skorohod integrals; then we use Galerkin approximations to show that the solutions of (\ref{1.a}) with deterministic initial value are Mallivin differentiable; finally, we prove Theorem \ref{56.I}.
For simplicity, we sometimes omit the parameter $\omega$ in the following when it is clear from the context.
\begin{lem}\label{39.1.I}
Let $G,K$ be real separable Hilbert spaces, $U$ is a subspace of $G$ and contains an orthonormal basis $\{e_i\}_{i=1}^{\infty}$ of $G$.
Suppose that a random variable $\eta$ takes values in $U$ 
and $\eta\in\mathcal{D}^{1,p}(G)$, $p>1$, $\|\eta\|_G<\delta$. Consider a $K$-valued random field $u=\{u(f):f\in G\}$ with continuously Fr\'{e}chet differentiable paths on $U$ (i.e. the map $G\ni f\mapsto u(f,\omega)\in K$ is continuously Fr\'{e}chet differentiable on $U\subset G$ for almost all $\omega\in\Omega$),
such that $u(f)\in\mathcal{D}^{1,r}(K)$, $r>1$, for any $f\in U$, and the Malliavin derivative $\mathcal{D}u(f)$ as a $L^2([0,T])\otimes K$-valued random field has a continuous version on $U$. Suppose we have
\begin{align*}
& E\bigg[\sup_{f\in U\cap B^{G}_{\delta}}\Big(\|u(f)\|_K^r+\|\mathcal{D}u(f)\|_{L^2([0,T])\otimes K}^r\Big)\bigg]<\infty ,\\
& E\bigg[\sup_{f\in U\cap B^{G}_{\delta}}\|\mathbb{D}u(f)\|_{L(G,K)}^q\bigg]<\infty ,
\end{align*}

\noindent where $B^{G}_{\delta}:=\{x\in G : \|x\|_G\leq \delta\}$, $1\leq q\leq \infty, \frac{1}{p}+\frac{1}{q}=\frac{1}{r}$.
Then $u(\eta)\in\mathcal{D}^{1,r}(K)$, and
\begin{align}\label{39.2.c}
\mathcal{D}\big(u(\eta)\big)=\mathbb{D}u(\eta)(\mathcal{D}\eta) +(\mathcal{D}u)(\eta) .
\end{align}
\end{lem}


\begin{proof}
Let $\{\rho_i\}_{i=1}^{\infty}$ be an orthonormal basis in $K$. Set $\eta_m:=\sum_{i=1}^m \langle \eta,e_i\rangle_G\,e_i$, and $u^n:=\sum_{j=1}^n \langle u,\rho_j\rangle_K\,\rho_j$.
Then by Lemma 3.2.3 in \cite{2006-Nualart-p-}, we have
\begin{align}\label{39.2.d}
\mathcal{D}\big(u^n(\eta_m)\big)=\mathbb{D}u^n(\eta_m)(\mathcal{D}\eta_m) +(\mathcal{D}u^n)(\eta_m) .
\end{align}

\noindent Letting $n,m \rightarrow\infty$, we can show that the terms on the right of (\ref{39.2.d}) converges to the corresponding terms in (\ref{39.2.c}). Since the Malliavin derivative operator $\mathcal{D}$ is closed, we conclude that $u(\eta)\in\mathcal{D}^{1,r}(K)$ and (\ref{39.2.c}) holds.
\end{proof}

Next, we establish a precise product rule for the indefinite Skorohod integrals under very weak conditions, this formula is the main tool used in the proof of Theorem \ref{56.I}.
\begin{prp}\label{39.5.I}
Let $G$ be a real seperable Hilbert space, Set $G^1=G$, $G^2=\mathbb{R}$. Consider processes of the form,
\begin{align}\label{39.5.1}
X^i_t=X^i_0+\int_0^t u^i_s\,dW_s+\int_0^t v^i_s\,ds,\quad i=1,2,
\end{align}
where 
$u^i\in\mathcal{L}^{1,2}_{loc}(G^i)$, $v^i$ is $G^i$-valued jointly measurable and $\int_0^T\|v^i_s\|_{G^i}\,ds<\infty$ a.s. $\omega\in\Omega$,
$X^i\in\mathcal{L}^{1,2}_{1,loc}(G^i)$
and $X^i_t\in\mathcal{D}^{1,2}_{loc}(G^i)$ for all $t\in[0,T]$,
$X^i$ and $u^i$ have versions which are $G^i$-valued continuous, then we have for any $t\in[0,T]$,
\begin{align}\label{39.5.2}
\begin{aligned}
X^1_t X^2_t=& X^1_0 X^2_0 +\int_0^t X^2_s u^1_s\,dW_s +\int_0^t X^2_s v^1_s\,ds +\int_0^t X^1_s u^2_s\,dW_s +\int_0^t X^1_s v^2_s\,ds \\
&+ \frac{1}{2}\int_0^t(\nabla X^1)_s u^2_s\,ds +\frac{1}{2}\int_0^t(\nabla X^2)_s u^1_s\,ds .
\end{aligned}
\end{align}
Moreover, $X^1 X^2\in\mathcal{L}^{1,2}_{1,loc}(G)$ and
\begin{align}\label{39.5.7}
(\nabla (X^1 X^2))_s=X^2_s(\nabla  X^1)_s+X^1_s(\nabla  X^2)_s .
\end{align}
\end{prp}

\begin{remark}\label{product rule a.s. t}
$X^i\in\mathcal{L}^{1,2}_{1,loc}(G^i)$ implies that $X^i_t\in\mathcal{D}^{1,2}_{loc}(G^i)$ for a.s. $t\in[0,T]$. Therefore, without the condition $X^i_t\in\mathcal{D}^{1,2}_{loc}(G^i)$ for all $t\in[0,T]$, (\ref{39.5.2}) holds only for a.s. $t\in[0,T]$.
\end{remark}

\begin{proof}
We first use a localization argument to assume that 
$u^i\in\mathcal{L}^{1,2}(G^i)$,
$X^i\in\mathcal{L}^{1,2}_{1}(G^i)$, $\sup_{0\leq t\leq T}\|X^i_t\|_{G^i}\leq k$, $\sup_{0\leq t\leq T}\|u^i_t\|_{G^i}\leq k$, $\int_0^T \|v^i_s\|_{G^i}\,ds\leq k$, for some fixed $k\in\mathbb{N}$. And also, for any fixed $t>0$, let $\{0=t_0^n\leq t_1^n\leq\cdots\leq t_{k_n}^n=t\}_{n\geq 1}$ be a sequence of partitions of the interval $[0,t]$ such that $\tau^n=\max_{0\leq j\leq k_n}(t_{j+1}^n-t_j^n)\rightarrow 0$ as $n\rightarrow\infty$ and $X^i(t^n_j)\in\mathcal{D}^{1,2}(G^i)$ for each $j=1,...,k_n$ and each $n\in\mathbb{N}$. Then we note the identities:
\begin{align}
  X^1_t X^2_t=& X^1_0 X^2_0 +\sum_{j=0}^{k_n}X^2(t^n_j)(X^1(t^n_{j+1})-X^1(t^n_j)) +\sum_{j=0}^{k_n}X^1(t^n_j)(X^2(t^n_{j+1})-X^2(t^n_j)) \nonumber\\ &+\sum_{j=0}^{k_n}(X^1(t^n_{j+1})-X^1(t^n_j))(X^2(t^n_{j+1})-X^2(t^n_j)), \label{39.5.3} \\
  X^1_t X^2_t=& X^1_0 X^2_0 +\sum_{j=0}^{k_n}X^2(t^n_{j+1})(X^1(t^n_{j+1})-X^1(t^n_j)) +\sum_{j=0}^{k_n}X^1(t^n_{j+1})(X^2(t^n_{i+1})-X^2(t^n_j)) \nonumber\\ &-\sum_{j=0}^{k_n}(X^1(t^n_{j+1})-X^1(t^n_j))(X^2(t^n_{j+1})-X^2(t^n_j)) \label{39.5.4},
\end{align}
by the similar steps 1--5 as Theorem 3.2.2 in \cite{2006-Nualart-p-}, we obtain the following formula from (\ref{39.5.3}),
\begin{align}\label{39.5.5}
\begin{aligned}
X^1_t X^2_t=& X^1_0 X^2_0 +\int_0^t X^2_s u^1_s\,dW_s +\int_0^t X^2_s v^1_s\,ds +\int_0^t X^1_s u^2_s\,dW_s +\int_0^t X^1_s v^2_s\,ds \\
&+\int_0^t u^1_s u^2_s\,ds +\int_0^t(\mathcal{D}^{-} X^2)_s u^1_s\,ds +\int_0^t(\mathcal{D}^{-} X^1)_s u^2_s\,ds .
\end{aligned}
\end{align}

\noindent Similarly, it follows from (\ref{39.5.4}) that
\begin{align}\label{39.5.6}
\begin{aligned}
X^1_t X^2_t=& X^1_0 X^2_0 +\int_0^t X^2_s u^1_s\,dW_s +\int_0^t X^2_s v^1_s\,ds +\int_0^t X^1_s u^2_s\,dW_s +\int_0^t X^1_s v^2_s\,ds \\
&-\int_0^t u^1_s u^2_s\,ds +\int_0^t(\mathcal{D}^{+} X^2)_s u^1_s\,ds +\int_0^t(\mathcal{D}^{+} X^1)_s u^2_s\,ds  .
\end{aligned}
\end{align}
Adding (\ref{39.5.5}) and (\ref{39.5.6}) and noticing that $(\nabla X)_s:=(\mathcal{D}^{+}X)_s + (\mathcal{D}^{-}X)_s$, we obtain (\ref{39.5.2}).
Obviously, $X^1_t X^2_t$ is Malliavin differentiable and $\mathcal{D}_s(X^1_t X^2_t)=X^2_t \mathcal{D}_s X^1_t + X^2_t \mathcal{D}_s X^1_t$, so it is easy to see that $X^1 X^2\in\mathcal{L}^{1,2}(G)$.
\begin{align*}
\begin{aligned}
& \int_0^T\sup_{s<t\leq(s+\frac{1}{n})\wedge T}E\|X^2_t\mathcal{D}_s X^1_t + X^1_t \mathcal{D}_s X^2_t- X^2_s(\mathcal{D}^{+} X^1)_s - X^1_s(\mathcal{D}^{+} X^2)_s\|_{G}\,ds \\
=& \int_0^T\sup_{s<t\leq(s+\frac{1}{n})\wedge T}E\|X^2_t(\mathcal{D}_s X^1_t-(\mathcal{D}^{+} X^1)_s)\|_{G}\,ds +\int_0^T\sup_{s<t\leq(s+\frac{1}{n})\wedge T}E\|(X^2_t-X^2_s)(\mathcal{D}^{+} X^1)_s\|_{G}\,ds  \\
& +\int_0^T\sup_{s<t\leq(s+\frac{1}{n})\wedge T}E\|X^1_t(\mathcal{D}_s X^2_t-(\mathcal{D}^{+} X^2)_s)\|_{G}\,ds +\int_0^T\sup_{s<t\leq(s+\frac{1}{n})\wedge T}E\|(X^1_t-X^1_s)(\mathcal{D}^{+} X^2)_s\|_{G}\,ds \\
:=& I_1+I_2+I_3+I_4 .
\end{aligned}
\end{align*}
Since $\sup_{0\leq t\leq T}\|X^i_t\|_{G^i}\leq k$ and $X^i\in\mathcal{L}^{1,2}_{1}(G^i)$, we have $I_1\rightarrow 0$ as $n\rightarrow\infty$.
\begin{align*}
I_2\leq E\int_0^T\sup_{s<t\leq(s+\frac{1}{n})\wedge T}|X^2_t-X^2_s|\|(\nabla  X^1)_s\|_{G}\,ds ,
\end{align*}
by the continuity of $X^2$ and the dominated convergence theorem, it follows that $I_2\rightarrow 0$ as $n\rightarrow\infty$. $I_3$ and $I_4$ also tend to zero by the same reason as $I_1$ and $I_2$. Therefore, we have
\begin{align*}
(\mathcal{D}^{+}(X^1 X^2))_s=X^2_s(\mathcal{D}^{+}X^1)_s+X^1_s(\mathcal{D}^{+}X^2)_s .
\end{align*}
Similarly, we have
\begin{align*}
(\mathcal{D}^{-}(X^1 X^2))_s=X^2_s(\mathcal{D}^{-}X^1)_s+X^1_s(\mathcal{D}^{-}X^2)_s .
\end{align*}
Hence, we obtain (\ref{39.5.7}).
\end{proof}


Let $Q(t):=\exp\{\sigma W(t)\}$. Consider the following system for each fixed $\omega\in\Omega$,
\begin{align}\label{4.b}
\left\{
\begin{aligned}
& dv(t,f)=-\nu \widehat{A}v(t,f)\,dt-Q(t)\widehat{B}\big(v(t,f)\big)\,dt+\frac{1}{Q(t)}\widehat{F}\big(Q(t)v(t,f),t\big)\,dt, \quad 0<t\leq T, \\
& v(0,f)=f\quad \text{in}\ \mathbb{W}.
\end{aligned}
\right.
\end{align}

\noindent The following lemma is taken from Propositin 4.1, 4.4 and 4.5 in \cite{2017-Shang-p-a}.
\begin{lem}\label{17.I-2}
Assume that {\bf(F1)} and {\bf(F2)} are satisfied, then for any $f\in\mathbb{W}$, a.s. $\omega\in\Omega$, there exists a unique solution to (\ref{4.b}). Furthermore, the solution map $[0,T]\times \mathbb{W}\times \Omega \ni(t,f,\omega)\mapsto v(t,f,\omega)\in \mathbb{W}$
is $\mathcal{B}([0,T])\otimes \mathcal{B}(\mathbb{W})\otimes \mathcal{F}/\mathcal{B}(\mathbb{W})$-measurable and $\mathcal{F}_t$-adapted, and
\begin{align}
\label{W norm estimate of v-2} &\sup_{t\in[0,T]}\big|v(t,f,\omega)\big|_{\mathbb{W}}^2\leq C(T) |f|_{\mathbb{W}}^2.
\end{align}
 Moreover, for a.s. $\omega\in\Omega$, $\forall\,t\in[0,T]$, the map $v(t,\cdot,\omega): \mathbb{W}\ni f\longmapsto v(t,f,\omega)\in\mathbb{V}$ is continuously Fr\'{e}chet differentiable on $\mathbb{W}\cap \mathbb{H}^4(\mathcal{O})$,
and the following estimate holds
\begin{align}\label{V and W norm estimate of z-2}
\begin{aligned}
\big\|\mathbb{D}v(t,f,\omega)(g)\big\|_{C([0,T];\mathbb{V})}^2\leq C(T,\|Q\|_{\infty,T},|f|_{\mathbb{W}})|g|_{\mathbb{W}}^2,
\end{aligned}
\end{align}
where $\|Q\|_{\infty,T}:=\sup_{0\leq t\leq T}Q(t)<\infty$ for a.s. $\omega\in\Omega$.
\end{lem}

\vskip 0.3cm

By the classical It\^{o}'s formula, we easily see that $Q(t)v(t,f)$ is a version of $u(t,f)$, where $u(t,f)$ is the solution of (\ref{Abstract}) with deterministic initial value $u(0)=f$.
Therefore, it is natural to ask whether $Q(t)v(t,\xi)$ is a solution of (\ref{Abstract}) or not.
In fact, the answer is affirmative.
To illustrate this, by Lemma \ref{39.1.I} and Proposition \ref{39.5.I} it is necessary to show that $v(t,f)\in\mathcal{D}_{loc}^{1,2}(\mathbb{V})$ and calculate $\mathcal{D}_r v(t,f)$ for $t\in[0,T]$. The uniqueness of solutions of (\ref{4.b}) implies that
\[v(t,f)=v^N(t,f)\quad \text{ on }\Omega_N:=\Big\{\omega: \sup\limits_{0\leq s\leq T}|W(s,\omega)|\leq N\Big\},\]
where $v^N(t,f)$ is the solution of an equation similar to (\ref{4.b}) only replacing $Q(s)$ by
\[Q^N(t):=\exp\big(\sigma[(-N)\vee W(t)\wedge N]\big) .\]

\noindent Thus it suffice to prove that $v^N(t,f)\in\mathcal{D}^{1,2}(\mathbb{V})$ for each fixed $N$. For this reason, we assume implicitly in the rest of this section that $Q=Q^N$. Noting that in this case
\begin{align*}
&&&\begin{aligned}
\|Q\|_{\infty,T}:=\sup_{0\leq t\leq T}Q(t)<\exp(\sigma N),
&&&\end{aligned}\\
&&&\begin{aligned}
\mathcal{D}_r Q(s) = \left\{
\begin{aligned}
& \sigma\exp\big(\sigma W(s)\big)I_{[0,s]}(r)\quad &&\text{on}\ \Omega_N ,\\
& 0 &&\text{on}\ \Omega\backslash\Omega_N ,
\end{aligned}
\right.
&&&\end{aligned}\\
&&&\begin{aligned}
\|\mathcal{D}Q\|_{\infty,T}:=\sup_{0\leq r\leq T}\sup_{0\leq s\leq T}|\mathcal{D}_r Q(s)|\leq \sigma\exp(\sigma N).
&&&\end{aligned}
\end{align*}


\vskip 0.3cm
To show that $v(t,f)$ is Malliavin differentiable, we appeal to Galerkin approximations. From (\ref{Basis}) we know that $\{\sqrt{\lambda_i}e_i\}$ is an orthonormal basis of $\mathbb{V}$. Let $\Pi_n$ be defined by
\[\Pi_ng:=\sum_{i=1}^n\lambda_i\langle g,e_i\rangle e_i,\quad \forall\, g\in\mathbb{W}^*.\]
For any integer $n\geq1$, Lemma 4.1 in \cite{2017-Shang-p-a} show that there exists a unique global solution to the following finite dimensional equation
\begin{align}\label{5.a}
\left\{
\begin{aligned}
& dv_n(t)= \Pi_n\Big[-\nu \widehat{A}v_n(t)-Q(t)\widehat{B}\big(v_n(t)\big)+\frac{1}{Q(t)}\widehat{F}\big(Q(t)v_n(t),t\big)\Big]\,dt,\quad t\in[0,T], \\
& v_n(0)=f_n:\triangleq\Pi_n f .
\end{aligned}
\right.
\end{align}
%
We also need the following two lemmas.
\begin{lem}\label{Lemma 41.I}
Assume that {\bf(F1)} holds, $v_n(t,f_n)$ is the solution of the equation (\ref{5.a}), then
\begin{align*}
  &\lim_{n\rightarrow\infty}E\int_0^T\left|v_n(t,f_n)-v(t,f)\right|_{\mathbb{W}}^2\,dt=0 ,\\
  &\lim_{n\rightarrow\infty}E\left|v_n(t,f_n)-v(t,f)\right|_{\mathbb{W}}^2=0,\quad\forall\, t\in[0,T].
\end{align*}
%
\end{lem}

\begin{proof}
In fact, the proof of Proposition 4.1 in \cite{2017-Shang-p-a} implies that for a.s. $\omega\in\Omega$,
\begin{align}\label{v_n weakly convergent in W}
v_n(t,f_n)\rightharpoonup v(t,f) \,\quad\text{weakly convergent in }\mathbb{W},\ \forall\, t\in[0,T].
\end{align}
(4.8--4.9) in \cite{2017-Shang-p-a} imply that the following energy equation for $v_n(t,f_n)$ holds:
\begin{align}\label{energy equation}
  |v_n(t,f_n)|_{\mathbb{W}}^2=|f_n|_{\mathbb{W}}^2\mathrm{e}^{-\frac{2\nu}{\alpha}t}+2\int_0^{t}K(v_n(s,f_n),s)\mathrm{e}^{-\frac{2\nu}{\alpha}(t-s)}\,ds,
\end{align}
where
\begin{align*}
  K(v_n(s,f_n),s):=\Big(\frac{\nu}{\alpha}{\rm curl}\big(v_n(s,f_n)\big)+{\rm curl}\big(F_Q(v_n(s,f_n),s)\big),{\rm curl}\big(v_n(s,f_n)-\alpha\Delta v_n(s,f_n)\big)\Big).
\end{align*}
The convergence (4.13--4.15) in \cite{2017-Shang-p-a} also allow us to pass to the limit in (\ref{energy equation}) to obtain that
\begin{align}\label{energy equation for v}
  \lim_{n\rightarrow\infty}|v_n(t,f_n)|_{\mathbb{W}}^2=|f|_{\mathbb{W}}^2\mathrm{e}^{-\frac{2\nu}{\alpha}t}+2\int_0^{t}K(v(s,f),s)\mathrm{e}^{-\frac{2\nu}{\alpha}(t-s)}\,ds.
\end{align}
By Theorem 4.1.2 in \cite{1998-Moise-p1369-1369}, we see that the right of (\ref{energy equation for v}) is just the energy equation for $v(t,f)$. Hence,
\begin{align*}
  \lim_{n\rightarrow\infty}|v_n(t,f_n)|_{\mathbb{W}}^2=|v(t,f)|_{\mathbb{W}}^2,
\end{align*}
which together with (\ref{v_n weakly convergent in W}) yield for a.s. $\omega\in\Omega$,
\begin{align}\label{v_n strongly convergent in W}
  v_n(t,f_n)\rightarrow v(t,f) \,\quad\text{strongly convergent in }\mathbb{W},\ \forall\,t\in[0,T].
\end{align}
Therefore, by (\ref{W norm estimate of v-2}) and the dominated convergence theorem, Lemma \ref{Lemma 41.I} follows immediately.
\end{proof}


Let $f\in\mathbb{W}$, $v(t,f)$ be the solution of (\ref{4.b}). Consider the following random evolution equation:
\begin{align}\label{43.a}
Y_r(t,f)
=& -\nu\int_0^t\widehat{A}Y_r(s,f)\,ds  -\int_0^t\mathcal{D}_r Q(s)\widehat{B}\big(v(s,f),v(s,f)\big)\,ds \nonumber\\
& -\int_0^t Q(s)\widehat{B}\big(Y_r(s,f),v(s,f)\big)\,ds  -\int_0^t Q(s)\widehat{B}\big(v(s,f),Y_r(s,f)\big)\,ds \nonumber\\
& +\int_0^t \mathcal{D}_r \left(\frac{1}{Q(s)}\right)\widehat{F}\big(Q(s)v(s,f),s\big)\,ds \nonumber\\
& +\int_0^t \frac{1}{Q(s)}\mathbb{D}\widehat{F}\big(Q(s)v(s,f),s\big)v(s,f)\mathcal{D}_r Q(s)\,ds \nonumber\\
& +\int_0^t \mathbb{D}\widehat{F}\big(Q(s)v(s,f),s\big)Y_r(s,f)\,ds .
\end{align}

\begin{lem}\label{Lemma 43.I}
Assume that {\bf(F1)} and {\bf(F2)} hold, then for each $f\in\mathbb{W}\cap \mathbb{H}^4(\mathcal{O}), r\in[0,T]$, there exists a unique solution
$Y_r(\cdot\,,f)\in C\big([0,T];\mathbb{V}\big)\cap L^{\infty}\big([0,T];\mathbb{W}\big)$ to the equation (\ref{43.a}). Moreover, the following estimates hold:
\begin{align*}
&\sup_{t\in[0,T]}\big|Y_r(t,f)\big|_{\mathbb{V}}^2\leq C(|f|_{\mathbb{W}},T,N),\quad \forall\,r\in[0,T], \\
&\sup_{t\in[0,T]}\big|Y_r(t,f)\big|_{\mathbb{W}}^2\leq C(|f|_{\mathbb{H}^4(\mathcal{O})},|f|_{\mathbb{W}},T,N),\quad \forall\,r\in[0,T].
\end{align*}
\end{lem}

\begin{proof}
The proof of this lemma is similar to the proof of Proposition 4.1, Proposition 4.3 and Proposition 4.4 in \cite{2017-Shang-p-a}, so we omit the details..
\end{proof}

\begin{prp}\label{40.I}
Assume that {\bf(F1)} and {\bf(F2)} hold, then for each $f\in\mathbb{W}\cap \mathbb{H}^4(\mathcal{O})$, $t\in[0,T]$,
the solution $v(t,f)$ of the equation (\ref{4.b}) is Malliavin differentiable as a $\mathbb{V}$-valued random variable,
and its Malliavin derivative $\mathcal{D}_r v(t,f)$ solves (\ref{43.a}) for all $t\in[0,T]$, a.s..
\end{prp}

\begin{proof}
Let $v_n(t,f_n)$ be the solution of the finite-dimensional random ordinary differential equation (\ref{5.a}), it is known(see e.g. \cite{2006-Nualart-p-}) that
$v_n$ is Malliavin differentiable and the corresponding Malliavin derivative $\mathcal{D}_r v_n(t,f_n)$ satisfies the following random ODE:
\begin{align}\label{42.a}
\mathcal{D}_r v_n(t,f_n)
=& -\nu\int_0^t\widehat{A}\mathcal{D}_r v_n(s,f_n)\,ds  -\int_0^t\mathcal{D}_r Q(s)\widehat{B}\big(v_n(s,f_n),v_n(s,f_n)\big)\,ds \nonumber\\
& -\int_0^t Q(s)\widehat{B}\big(\mathcal{D}_r v_n(s,f_n),v_n(s,f_n)\big)\,ds  -\int_0^t Q(s)\widehat{B}\big(v_n(s,f_n),\mathcal{D}_r v_n(s,f_n)\big)\,ds \nonumber\\
& +\int_0^t \mathcal{D}_r \left(\frac{1}{Q(s)}\right)\widehat{F}\big(Q(s)v_n(s,f_n),s\big)\,ds \nonumber\\
& +\int_0^t \frac{1}{Q(s)}\mathbb{D}\widehat{F}\big(Q(s)v_n(s,f_n),s\big)v_n(s,f_n)\mathcal{D}_r Q(s)\,ds \nonumber\\
& +\int_0^t \mathbb{D}\widehat{F}\big(Q(s)v_n(s,f_n),s\big)\mathcal{D}_r v_n(s,f_n)\,ds ,
\end{align}
for all $t\in[0,T]$. Since the Malliavin derivative operator $\mathcal{D}$ is closed, in view of Lemma \ref{Lemma 41.I}, to prove the Proposition \ref{40.I} it suffice to show that
\begin{align*}
\lim_{n\rightarrow\infty}E\left[\sup_{0\leq r\leq T}\sup_{0\leq t\leq T}\big|\mathcal{D}_r v_n(t,f_n)-Y_r(t,f)\big|_{\mathbb{V}}^2\right]=0 .
\end{align*}

\noindent From (\ref{42.a}) and (\ref{43.a}), it follows that
%
\begin{align}\label{50.a}
&\ \frac{1}{2}\left|\mathcal{D}_r v_n(t,f_n)-Y_r(t,f)\right|_{\mathbb{V}}^2 \nonumber\\
=&\,-\nu\int_0^t\left\|\mathcal{D}_r v_n(s,f_n)-Y_r(s,f)\right\|^2\,ds \nonumber\\
&\,-\int_0^t \mathcal{D}_r Q(s)\big\langle \widehat{B}\big(v_n(s,f_n),v_n(s,f_n)\big)-\widehat{B}\big(v(s,f),v(s,f)\big), \mathcal{D}_r v_n(s,f_n)-Y_r(s,f)\big\rangle\,ds \nonumber\\
&\,-\int_0^t Q(s)\big\langle \widehat{B}\big(\mathcal{D}_r v_n(s,f_n),v_n(s,f_n)\big)-\widehat{B}\big(Y_r(s,f),v(s,f)\big), \mathcal{D}_r v_n(s,f_n)-Y_r(s,f)\big\rangle\,ds\nonumber\\
&\,-\int_0^t Q(s)\big\langle \widehat{B}\big(v_n(s,f_n),\mathcal{D}_r v_n(s,f_n)\big)-\widehat{B}\big(v(s,f),Y_r(s,f)\big),\mathcal{D}_r v_n(s,f_n)-Y_r(s,f)\big\rangle\,ds\nonumber\\
&\,+\int_0^t \mathcal{D}_r \left(\frac{1}{Q(s)}\right)\big\langle \widehat{F}\big(Q(s)v_n(s,f_n),s\big)-\widehat{F}\big(Q(s)v(s,f),s\big),\mathcal{D}_r v_n(s,f_n)-Y_r(s,f)\big\rangle\,ds\nonumber\\
&\,+\int_0^t \frac{\mathcal{D}_r Q(s)}{Q(s)}\big\langle \mathbb{D}\widehat{F}\big(Q(s)v_n(s,f_n),s\big)v_n(s,f_n)-\mathbb{D}\widehat{F}\big(Q(s)v(s,f),s\big)v(s,f),\nonumber\\
&~~~~~~~~~~~~~~~~~~~~~\mathcal{D}_r v_n(s,f_n)-Y_r(s,f)\big\rangle\,ds\nonumber\\
&\,+\int_0^t \big\langle \mathbb{D}\widehat{F}\big(Q(s)v_n(s,f_n),s\big)\mathcal{D}_r v_n(s,f_n)-\mathbb{D}\widehat{F}\big(Q(s)v(s,f),s\big)Y_r(s,f),\nonumber\\
&~~~~~~~~~~~\mathcal{D}_r v_n(s,f_n)-Y_r(s,f)\big\rangle\,ds \nonumber\\
:=&\,I_1+I_2+I_3+I_4+I_5+I_6+I_7.
\end{align}
Now we estimate these terms on the right of (\ref{50.a}).
\begin{align}\label{I_2}
\begin{aligned}
I_2=& -\int_0^t \mathcal{D}_r Q(s)\big\langle \widehat{B}\big(v_n(s,f_n)-v(s,f),v_n(s,f_n)\big), \mathcal{D}_r v_n(s,f_n)-Y_r(s,f)\big\rangle\,ds \\
& -\int_0^t \mathcal{D}_r Q(s)\big\langle \widehat{B}\big(v(s,f),v_n(s,f_n)-v(s,f)\big), \mathcal{D}_r v_n(s,f_n)-Y_r(s,f)\big\rangle\,ds \\
:=& I_{2a}+I_{2b}.
\end{aligned}
\end{align}
By (\ref{B inequalities}), we have
\begin{align}
\begin{aligned}\label{I_2a}
|I_{2a}|\leq\,& \|\mathcal{D}Q\|_{\infty,T}\int_0^t\big|v_n(s,f_n)-v(s,f)\big|_{\mathbb{W}} \big|v_n(s,f_n)\big|_{\mathbb{W}}\big|\mathcal{D}_r v_n(s,f_n)-Y_r(s,f)\big|_{\mathbb{V}}\,ds \\
\leq\,& \frac{1}{2}\|\mathcal{D}Q\|_{\infty,T}^2\int_0^t\big|v_n(s,f_n)-v(s,f)\big|_{\mathbb{W}}^2\,ds  +\frac{1}{2}\int_0^t\big|\mathcal{D}_r v_n(s,f_n)-Y_r(s,f)\big|_{\mathbb{V}}^2 \big|v_n(s,f_n)\big|_{\mathbb{W}}^2\,ds,
\end{aligned}\\
\begin{aligned}\label{I_2b}
|I_{2b}|\leq\,& \frac{1}{2}\|\mathcal{D}Q\|_{\infty,T}^2\int_0^t\big|v_n(s,f_n)-v(s,f)\big|_{\mathbb{W}}^2\,ds  +\frac{1}{2}\int_0^t\big|\mathcal{D}_r v_n(s,f_n)-Y_r(s,f)\big|_{\mathbb{V}}^2 \big|v(s,f)\big|_{\mathbb{W}}^2\,ds .
\end{aligned}
\end{align}

\begin{align}\label{I_3}
\begin{aligned}
I_3=& -\int_0^t Q(s)\big\langle \widehat{B}\big(Y_r(s,f),v_n(s,f_n)-v(s,f)\big),\mathcal{D}_r v_n(s,f_n)-Y_r(s,f)\big\rangle\,ds \\
& -\int_0^t Q(s)\big\langle \widehat{B}\big(\mathcal{D}_r v_n(s,f_n)-Y_r(s,f)),v_n(s,f_n)\big),\mathcal{D}_r v_n(s,f_n)-Y_r(s,f)\big\rangle\,ds \\
:=& I_{3a}+I_{3b} ,
\end{aligned}
\end{align}

\noindent where
\begin{align}
&\begin{aligned}\label{I_3a}
|I_{3a}|\leq\,& \|Q\|_{\infty,T}\int_0^t\big|Y_r(s,f)\big|_{\mathbb{W}} \big|\mathcal{D}_r v_n(s,f_n)-Y_r(s,f)\big|_{\mathbb{V}}\big|v_n(s,f_n)-v(s,f)\big|_{\mathbb{W}}\,ds \\
\leq\,& \frac{1}{2}\|Q\|_{\infty,T}^2\int_0^t\big|v_n(s,f_n)-v(s,f)\big|_{\mathbb{W}}^2\,ds  +\frac{1}{2}\int_0^t\big|\mathcal{D}_r v_n(s,f_n)-Y_r(s,f)\big|_{\mathbb{V}}^2 \big|Y_r(s,f)\big|_{\mathbb{W}}^2\,ds,
&\end{aligned}\\
&\begin{aligned}\label{I_3b}
|I_{3b}|\leq C\|Q\|_{\infty,T}\int_0^t\big|\mathcal{D}_r v_n(s,f_n)-Y_r(s,f)\big|_{\mathbb{V}}^2 \big|v_n(s,f_n)\big|_{\mathbb{W}}\,ds.
&\end{aligned}
\end{align}

\noindent In the same way, we have
\begin{align}\label{I_4}
\begin{aligned}
I_4=& -\int_0^t Q(s)\big\langle \widehat{B}\big(v_n(s,f_n)-v(s,f),Y_r(s,f)\big),\mathcal{D}_r v_n(s,f_n)-Y_r(s,f)\big\rangle\,ds \\
& -\int_0^t Q(s)\big\langle \widehat{B}\big(v_n(s,f_n),\mathcal{D}_r v_n(s,f_n)-Y_r(s,f)\big),\mathcal{D}_r v_n(s,f_n)-Y_r(s,f)\big\rangle\,ds \\
:=& I_{4a}+I_{4b} ,
\end{aligned}
\end{align}
and
\begin{align}\label{I_4a}
\begin{aligned}
|I_{4a}|\leq \|Q\|_{\infty,T}\int_0^t\big|v_n(s,f_n)-v(s,f)\big|_{\mathbb{W}} \big|\mathcal{D}_r v_n(s,f_n)-Y_r(s,f)\big|_{\mathbb{V}}\big|Y_r(s,f)\big|_{\mathbb{W}}\,ds,
\end{aligned}
\end{align}
obviously, $|I_{4a}|$ has the same estimate as $|I_{3a}|$, and $I_{4b}=0$ due to (\ref{B inequalities}).
Note that
\begin{align*}
\left|\mathcal{D}_r\left(\frac{1}{Q(s)}\right)\right|\left|Q(s)\right|=\left|\frac{\mathcal{D}_r Q(s)}{Q(s)}\right|\leq \sigma,
\end{align*}
thus, by {\bf (F1)} we have
\begin{align}\label{I_5}
\begin{aligned}
|I_5|\leq &\, C\int_0^t \left|\mathcal{D}_r\left(\frac{1}{Q(t)}\right)\right|\left| Q(t)\right|\big|v_n(s,f_n)-v(s,f)\big|_{\mathbb{V}}
\big|\mathcal{D}_r v_n(s,f_n)-Y_r(s,f)\big|_{\mathbb{V}}\,ds \\
\leq &\, C\int_0^t\big|v_n(s,f_n)-v(s,f)\big|_{\mathbb{V}}^2\,ds +C\int_0^t\big|\mathcal{D}_r v_n(s,f_n)-Y_r(s,f)\big|_{\mathbb{V}}^2\,ds.
\end{aligned}
\end{align}
The term $I_6$ can be bounded as follows:
\begin{align}\label{I_6}
I_6=&\, \int_0^t \frac{\mathcal{D}_r Q(s)}{Q(s)}\Big( \mathbb{D}\widehat{F}\big(Q(s)v_n(s,f_n),s\big)\big(v_n(s,f_n)-v(s,f)\big),\mathcal{D}_r v_n(s,f_n)-Y_r(s,f)\Big)_{\mathbb{V}}\,ds\nonumber\\
&\, +\int_0^t \frac{\mathcal{D}_r Q(s)}{Q(s)}\Big(\big[ \mathbb{D}\widehat{F}\big(Q(s)v_n(s,f_n),s\big)-\mathbb{D}\widehat{F}\big(Q(s)v(s,f),s\big)\big]v(s,f), \nonumber\\
&~~~~~~~~~~~~~~~~~~~~~\mathcal{D}_r v_n(s,f_n)-Y_r(s,f)\Big)_{\mathbb{V}}\,ds \nonumber\\
:=&\, I_{6a}+I_{6b}.
\end{align}

\noindent {\bf(F1)} and {\bf(F2)} imply that
\begin{align*}
\big\|\mathbb{D}F\big(Q(s)v(s,f),s\big)\big\|_{L(\mathbb{V})}\leq C, \quad \forall\, s\in[0,T]  .
\end{align*}
Hence,
\begin{align}\label{I_6a}
\begin{aligned}
|I_{6a}|\leq  C\int_0^t \left|\frac{\mathcal{D}_r Q(s)}{Q(s)}\right|\big|v_n(s,f_n)-v(s,f)\big|_{\mathbb{V}}
\big|\mathcal{D}_r v_n(s,f_n)-Y_r(s,f)\big|_{\mathbb{V}}\,ds,
\end{aligned}
\end{align}
so $I_{6a}$ has the same estimate as $I_5$.
\begin{align}\label{I_6b}
\begin{aligned}
|I_{6b}|\leq &\, \int_0^t \Psi(n,s)\times\big|v(s,f)\big|_{\mathbb{V}}
\big|\mathcal{D}_r v_n(s,f_n)-Y_r(s,f)\big|_{\mathbb{V}}\,ds \\
\leq &\,\frac{1}{2}\int_0^t\Psi^2(n,s)\,ds +\frac{1}{2}\int_0^t
\big|\mathcal{D}_r v_n(s,f_n)-Y_r(s,f)\big|_{\mathbb{V}}^2\big|v(s,f)\big|_{\mathbb{V}}^2\,ds ,
\end{aligned}
\end{align}
where
\[
\Psi(n,s)=\left\|\mathbb{D}\widehat{F}\big(Q(s)v_n(s,f_n),s\big) -\mathbb{D}\widehat{F}\big(Q(s)v(s,f),s\big)\right\|_{L(\mathbb{V})} .
\]
Similarly,
\begin{align}\label{I_7}
\begin{aligned}
I_7=&\, \int_0^t \Big( \mathbb{D}\widehat{F}\big(Q(s)v_n(s,f_n),s\big)\big(\mathcal{D}_r v_n(s,f_n)-Y_r(s,f)\big), \mathcal{D}_r v_n(s,f_n)-Y_r(s,f)\Big)_{\mathbb{V}}\,ds\\
&\, +\int_0^t \Big(\big[ \mathbb{D}\widehat{F}\big(Q(s)v_n(s,f_n),s\big)-\mathbb{D}\widehat{F}\big(Q(s)v(s,f),s\big)\big]Y_r(s,f), \\
&~~~~~~~~~\mathcal{D}_r v_n(s,f_n)-Y_r(s,f)\Big)_{\mathbb{V}}\,ds \\
:=&\, I_{7a}+I_{7b},
\end{aligned}
\end{align}

\noindent where
\begin{align}
&\begin{aligned}\label{I_7a}
|I_{7a}|\leq  C\int_0^t \big|\mathcal{D}_r v_n(s,f_n)-Y_r(s,f)\big|_{\mathbb{V}}^2\,ds ,
&\end{aligned}\\
&\begin{aligned}\label{I_7b}
|I_{7b}|\leq &\, \int_0^t \Psi(n,s)\times\big|Y_r(s,f)\big|_{\mathbb{V}}\big|\mathcal{D}_r v_n(s,f_n)-Y_r(s,f)\big|_{\mathbb{V}}\,ds \\
\leq &\,\frac{1}{2}\int_0^t\Psi^2(n,s)\,ds +\frac{1}{2}\int_0^t
\big|\mathcal{D}_r v_n(s,f_n)-Y_r(s,f)\big|_{\mathbb{V}}^2\big|Y_r(s,f)\big|_{\mathbb{V}}^2\,ds .
&\end{aligned}
\end{align}

\noindent Substituting the above estimates (\ref{I_2}--\ref{I_7b}) into (\ref{50.a}) gives
\begin{align*}
&\ \frac{1}{2}\left|\mathcal{D}_r v_n(t,f_n)-Y_r(t,f)\right|_{\mathbb{V}}^2 +\nu\int_0^t\left\|\mathcal{D}_r v_n(s,f_n)-Y_r(s,f)\right\|^2\,ds \\
\leq &\, \int_0^t\Psi^2(n,s)\,ds+ \Big(C+\|\mathcal{D}Q\|^2_{\infty,T},\|Q\|^2_{\infty,T}\Big)\times\int_0^t\big|v_n(s,f_n)-v(s,f)\big|_{\mathbb{W}}^2\,ds\\
&\,+C\int_0^t\big|\mathcal{D}_r v_n(s,f_n)-Y_r(s,f)\big|_{\mathbb{V}}^2\Big(\big|v_n(t,f_n)\big|_{\mathbb{W}}^2 +\big|v(t,f)\big|^2_{\mathbb{W}} +\big|Y_r(t,f)\big|_{\mathbb{W}}^2 \\ &~~~~~~~~~+\|Q\|_{\infty,T}\big|v_n(t,f_n)\big|_{\mathbb{W}} +\big|v(t,f)\big|^2_{\mathbb{V}} +\big|Y_r(t,f)\big|_{\mathbb{V}}^2 +1\Big)\,ds .
\end{align*}

\noindent Applying Gronwall inequality, using Lemma \ref{Lemma 43.I} and Lemma \ref{17.I-2}, we obtain
\begin{align*}
&\ \sup_{0\leq t\leq T}\left|\mathcal{D}_r v_n(t,f_n)-Y_r(t,f)\right|_{\mathbb{V}}^2 +2\nu\int_0^T\left\|\mathcal{D}_r v_n(s,f_n)-Y_r(s,f)\right\|^2\,ds \\
\leq &\, C(|f|_{\mathbb{H}^4},|f|_{\mathbb{W}},T,N)\times\bigg(\int_0^T\Psi^2(n,s)\,ds +C(N)\int_0^T\big|v_n(s,f_n)-v(s,f)\big|_{\mathbb{W}}^2\,ds\bigg) .
\end{align*}

\noindent Due to Lemma \ref{Lemma 41.I}, (\ref{v_n strongly convergent in W}) and the continuity of $\mathbb{D}F$ in {\bf (F2)}, applying the dominated convergence theorem, we conclude
\begin{align*}
\lim_{n\rightarrow\infty}E\left[\sup_{0\leq r\leq T}\sup_{0\leq t\leq T}\big|\mathcal{D}_r v_n(t,f_n)-Y_r(t,f)\big|_{\mathbb{V}}^2\right]=0 .
\end{align*}

\end{proof}

\begin{remark}\label{remark 4.1}
In the same way, we can obtain the fact: for any sequence $h,h_n\in\mathbb{W}\cap \mathbb{H}^4(\mathcal{O})$, and $h_n\rightarrow h$ in $\mathbb{W}$-norm as $n\rightarrow\infty$, we have for a.s. $\omega\in\Omega$,
\[\lim_{n\rightarrow\infty}\sup_{0\leq r\leq T}\sup_{0\leq t\leq T}\left|Y_r(t,h_n)-Y_r(t,h)\right|_{\mathbb{V}}^2=0 .\]

\noindent Owing to this, it follows that for each $t\in[0,T]$,
\begin{align*}
\lim_{n\rightarrow\infty}\left|\mathcal{D}v(t,h_n)-\mathcal{D}v(t,h)\right|_{L^2([0,T])\otimes \mathbb{V}}^2
=&\,\lim_{n\rightarrow\infty}\int_0^T\left|\mathcal{D}_r v(t,h_n)-\mathcal{D}_r v(t,h)\right|_{\mathbb{V}}^2\,dr \\
\leq &\,T\lim_{n\rightarrow\infty}\sup_{0\leq r\leq T}\sup_{0\leq t\leq T}\left|\mathcal{D}_r v(t,h_n)-\mathcal{D}_r v(t,h)\right|_{\mathbb{V}}^2 \\
=&\,0 .
\end{align*}

\noindent Therefore, the $L^2([0,T])\otimes\mathbb{V}$-valued random field $\mathcal{D}v(t,h)$ has a continuous version.
\end{remark}

\begin{proof}[{\bf Proof of Theorem \ref{56.I}}]
%
Existence. It follows from Lemma \ref{39.1.I}, Lemma \ref{17.I-2}, Proposition \ref{40.I}, Lemma \ref{Lemma 43.I} and  Remark \ref{remark 4.1} that $v(t,\xi)\in\mathcal{D}^{1,2}_{loc}(\mathbb{V})$ and
\begin{align}\label{D X=}
\mathcal{D}_s v(t,\xi)=\mathbb{D}v(t,\xi)(\mathcal{D}_s\xi) +(\mathcal{D}_s v(t))(\xi),
\end{align}
for every $t\in[0,T]$. Since each term of (\ref{D X=}) is continuous in $t$, and $(\mathcal{D}_s v(t))(\xi)=0$ for any $t\leq s$, by a localization argument and the dominated convergence theorem, we get $v(\cdot,\xi)\in\mathcal{L}^{1,2}_{1,loc}(\mathbb{V})$.
On the other hand, obviously, $Q\in\mathcal{L}^{1,2}_{1}$, $(\nabla Q)_s=\sigma Q(s)$ and
\begin{align*}
Q(t)=1+\int_0^t\sigma Q(s)\,dW(s)+\int_0^t\frac{1}{2}\sigma^2 Q(s)\,ds .
\end{align*}
Note that
\begin{align}\label{v_xi expansion}
v(t,\xi)=\xi+\int_0^t\Big[-\nu \widehat{A}v(s,\xi)-Q(t)\widehat{B}\big(v(s,\xi)\big)+Q^{-1}(s)\widehat{F}\big(Q(s)v(s,\xi),s\big)\Big]\,dt, \quad t\in[0,T],
\end{align}
\noindent therefore, we can apply Proposition \ref{39.5.I} to obtain that $v(t,\xi)Q(t)$ is a solution of (\ref{Abstract}).

Uniqueness. Let $u$ be a solution of (\ref{Abstract}), define the process $v(t)=u(t)Q^{-1}(t)$, $t\in[0,T]$.
By (\ref{S-integral and I-integral}), Proposition \ref{39.5.I} and the continuity of $u$, we immediately get that $v$ satisfies the equation (\ref{v_xi expansion}) for a.s. $\omega\in\Omega$. Now uniqueness of solutions for the equation (\ref{Abstract}) follows easily from the uniqueness of solutions for the equation (\ref{v_xi expansion}).
\end{proof}

\section*{Acknowledgements}
The author sincerely thank Professor Tusheng Zhang and Jianliang Zhai for their instructions and many invaluable suggestions.


\begin{thebibliography}{10}

\bibitem{2016-Arada-p2557-2586}
N.~Arada.
\newblock On the convergence of the two-dimensional second grade fluid model to
  the Navier-Stokes equation.
\newblock {\em Journal of Differential Equations}, 260(3):2557--2586, 2016.

\bibitem{2003-Busuioc-p1119-1119}
A.~V. Busuioc and T.~S. Ratiu.
\newblock The second grade fluid and averaged Euler equations with Navier-slip
  boundary conditions.
\newblock {\em Nonlinearity}, 16(3):1119--1149, 2003.

\bibitem{1999-Busuioc-p1241-1246}
V.~Busuioc.
\newblock On second grade fluids with vanishing viscosity.
\newblock {\em Comptes Rendus de I'Acad\'{e}mie des Sciences-Series I - Mathematics}, 328(12):1241--1246, 1999.

\bibitem{1997-Cioranescu-p317-335}
D.~Cioranescu and V.~Girault.
\newblock Weak and classical solutions of a family of second grade fluids.
\newblock {\em International Journal of Non-Linear Mechanics}, 32(2):317--335,
  1997.

\bibitem{1984-Cioranescu-p178-197}
D.~Cioranescu and E.~H. Ouazar.
\newblock Existence and uniqueness for fluids of second grade.
\newblock {\em Nonlinear Partial Differential Equations}, 109:178--197, 1984.

\bibitem{1995-Dunn-p689-729}
J.~Dunn and K.~Rajagopal.
\newblock Fluids of differential type: critical review and thermodynamic
  analysis.
\newblock {\em International Journal of Engineering Science}, 33(5):689--729,
  1995.

\bibitem{1974-Dunn-p191-252}
J.~E. Dunn and R.~L. Fosdick.
\newblock Thermodynamics, stability, and boundedness of fluids of complexity 2
  and fluids of second grade.
\newblock {\em Archive for Rational Mechanics and Analysis}, 56(3):191--252,
  1974.

\bibitem{1979-Fosdick-p145-152}
R.~Fosdick and K.~Rajagopal.
\newblock Anomalous features in the model of ¡°second order fluids¡±.
\newblock {\em Archive for Rational Mechanics and Analysis}, 70(2):145--152,
  1979.



\bibitem{2013-Mohammed-p1380-1408}
S.~Mohammed and T.~Zhang.
\newblock Anticipating stochastic 2D Navier-Stokes equations.
\newblock {\em Journal of Functional Analysis}, 264(6):1380--1408, 2013.

\bibitem{1998-Moise-p1369-1369}
I.~Moise, R.~Rosa, and X.~Wang.
\newblock Attractors for non-compact semigroups via energy equations.
\newblock {\em Nonlinearity}, 11(5):1369--1393, 1998.

\bibitem{2006-Nualart-p-}
D.~Nualart.
\newblock {\em The Malliavin calculus and related topics}, volume 1995.
\newblock Springer, 2006.

\bibitem{2010-Razafimandimby-p1-47}
P.~A. Razafimandimby and M.~Sango.
\newblock Weak solutions of a stochastic model for two-dimensional second grade
  fluids.
\newblock {\em Boundary Value Problems}, 2010(1):1--47, 2010.

\bibitem{2012-Razafimandimby-p4251-4270}
P.~A. Razafimandimby and M.~Sango.
\newblock Strong solution for a stochastic model of two-dimensional second
  grade fluids: Existence, uniqueness and asymptotic behavior.
\newblock {\em Nonlinear Analysis: Theory, Methods $\&$ Applications},
  75(11):4251--4270, 2012.

\bibitem{2017-Shang-p-a}
S.~Shang.
\newblock Stochastic flows of two-dimensional second grade fluids.
\newblock {\em arXiv:1703.08821}.

\bibitem{2017-Shang-p-}
S.~Shang, J.~Zhai, and T.~Zhang.
\newblock Strong solutions for a stochastic model of 2-D second grade fluids
  driven by L\'{e}vy noise.
\newblock {\em arXiv:1701.00314}.


\bibitem{2016-Wang-p196-213}
R.~Wang, J.~Zhai, and T.~Zhang.
\newblock Exponential mixing for stochastic model of two-dimensional second
  grade fluids.
\newblock {\em Nonlinear Analysis: Theory, Methods $\&$ Applications},
  132:196--213, 2016.

\bibitem{2016-Zhai-p1-28}
J.~Zhai and T.~Zhang.
\newblock Large deviations for stochastic models of two-dimentional second grade fluids.
\newblock {\em Appl. Math. Optimiz.}, 75(3):471--498, 2017.



\end{thebibliography}

\end{document}